\documentclass[12pt]{article}

\usepackage{amsmath}
\usepackage{amssymb}
\usepackage{amsfonts}
\usepackage{amsthm} 
\usepackage{enumerate}
\usepackage{graphicx}

\newtheorem{theorem}{Theorem}[section]
\newtheorem{proposition}[theorem]{Proposition}
\newtheorem{lemma}[theorem]{Lemma}
\newtheorem{definition}[theorem]{Definition}

\newtheorem{conjecture}[theorem]{Conjecture}

\numberwithin{equation}{section}

\newcommand{\rr}{{\mathbb R}}
\newcommand{\zz}{{\mathbb Z}}
\newcommand{\nn}{{\mathbb N}}
\newcommand{\cc}{{\mathbb C}}

\newcommand{\qq}{{\mathbb Q}}

\newcommand{\one}{{\bf 1}}

\newcommand{\proj}{\text{proj}}

\newcommand{\supp}{\hbox{supp\,}}

\newcommand{\C}{\mathbb{C}}

\begin{document}

\title{Recent progress on Favard length estimates for planar Cantor sets}
\author{Izabella {\L}aba}
\date{December 1, 2012}
\maketitle

\allowdisplaybreaks{


\section{Introduction}

Let $E_\infty =\bigcup_{n-1}^\infty E_n$ be a self-similar Cantor set in the plane, constructed as a limit of Cantor iterations $E_n$. We will assume that $E_\infty$ has Hausdorff dimension 1. The {\em Favard length problem}, also known as {\em Buffon's needle problem} (after Comte de Buffon), concerns the average (with respect to the angle) length of linear projections of $E_n$. 
This turns out to be a fascinating and very difficult problem. Substantial progress on it was only achieved in the last few years, revealing connections to harmonic analysis, combinatorics and number theory. The purpose of this paper is to survey these developments with emphasis on presenting the main ideas in the simplest possible settings, often at the expense of generality.

The interest in such estimates is motivated by, on the one hand, considerations in ergodic theory, and on the other hand, questions in the theory of analytic functions. On the analytic side, Favard length estimates have been linked to the Painlev\'e problem in complex analysis (on a geometric characterization of removable sets for bounded analytic functions), through the work of Vitushkin, Mattila, Jones, David, Tolsa and others. (See \cite{To} for more details.) In ergodic theory, there is a large family of related conjectures and open problems concerning projections of fractal sets, Bernoulli convolutions, intersections and sumsets of fractal Cantor sets. For instance, a well known conjecture of Furstenberg states that all linear projections of the 4-corner set (defined below) corresponding to irrational slopes have Hausdorff dimension 1. This is still unsolved, but it is possible that progress on it could have impact on the Favard length problem, or vice versa.

We begin by giving a few examples of the Cantor sets under consideration.

\medskip\noindent
{\bf Example 1: The 4-corner set.} Divide the unit square $[0,1]^2$ into
16 congruent squares and keep the 4 squares on the corners, discarding the rest. This is
$K_1$. Repeat the procedure within each of the 4 selected squares 
to produce $K_2$, consisting of 16 squares. Continue the iteration
indefinitely. The resulting set $K=\bigcap_{n=1}^\infty K_n$ is a fractal self-similar set of dimension 1, also called the Garnett set, after John Garnett used it in the theory of analytic functions as an example of a set with positive 1-dimensional length and zero analytic capacity.

\begin{figure}[htbp]\centering\includegraphics[width=4.in]{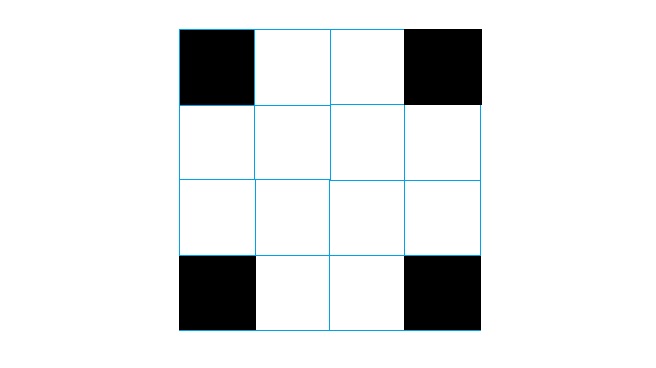}
\caption{The 4-corner set, !st iteration.}
\end{figure}

\begin{figure}[htbp]\centering\includegraphics[width=4.in]{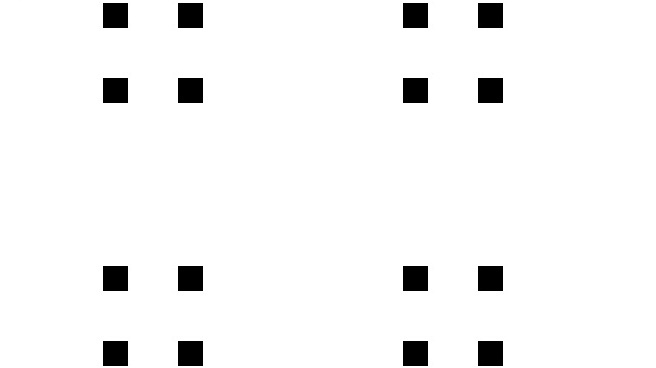}
\caption{The 4-corner set, 2nd iteration.}
\end{figure}

\medskip\noindent
{\bf Example 2: The Sierpi\'nski gasket.} The construction is similar, but we start with a triangle, divide each side in three equal parts, and keep the three triangles in the corners. (In the literature, the Sierpi\'nski gasket is often defined so that the three triangles have sidelengths equal to half, rather than a third, of the sidelength of the large triangle. The dimension of that set is $(\log 4)/(\log 3)$. We have modified the scaling so as to produce a gasket of dimension 1.)

\begin{figure}[htbp]\centering\includegraphics[width=4.in]{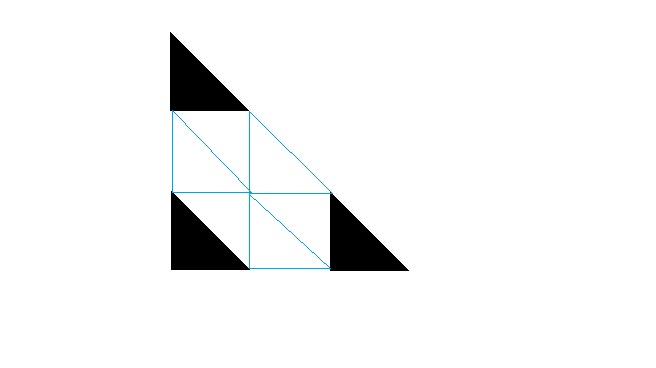}
\caption{The Sierpi\'nski gasket, 1st iteration.}
\end{figure}

\medskip\noindent
{\bf Example 3: Rational product Cantor sets.} Fix an integer $L\geq 4$. Divide the initial square into $L^2$ identical squares, fix $A,B\subset \{0,1,\dots,L-1\}$, let $E_1$ consist of those squares whose lower left vertices are $(a/L,b/L)$ for $a\in A$ and $b\in B$, then continue the iteration indefinitely. If $|A||B|=L$ (which we will assume from now on), the Cantor set 
$E=\bigcap_{n=1}^\infty E_n$ has dimension $1$. We will also assume that $|A|,|B|\geq 2$, to avoid degenerate cases. The 4-corner set is a special case of this, with $L=4$ and $A=B=\{0,3\}$. Another example is shown in Figure \ref{fig-product1}.

\begin{figure}[htbp]\centering\includegraphics[width=4.in]{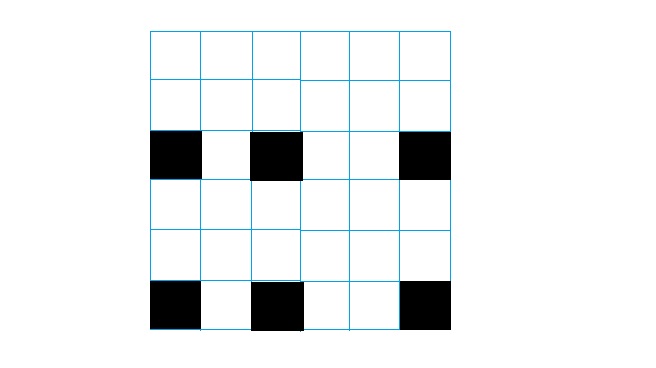}
\caption{The first iteration of a product Cantor set with $L=6,\ A=\{0,2,5\},\ B=\{0,3\}$.}
\label{fig-product1}
\end{figure}

\medskip\noindent
{\bf Example 4: General self-similar sets.}
Identify the plane with $\cc$, and let $T_1,\dots,T_L:\C\to\C$ be  similarity maps of the form $T_j(z)=\frac1{L}z+z_j$, where $z_1,\dots,z_L$ are distinct and not colinear. Then there is a unique compact set $E_\infty\subset\cc$ such that $E_\infty=\bigcup_{j=1}^L T_j(E_\infty)$. Under mild conditions (the {\em open set condition} of \cite{mattila} suffices), $E_\infty$ has Hausdorff dimension 1. Instead of Cantor iterations, one then considers the $L^{-n}$-neighbourhoods $E_n$ of $E_\infty$. 

\bigskip

We now proceed to the rigorous statement of the problem. 
Let $\proj_\theta(x,y)=x\cos\theta
+y\sin\theta$ be the projection of $(x,y)$ on a line making an angle $\theta$ with the positive 
$x$-axis (all projections are treated as subsets of $\rr$). 

\begin{definition}\label{def-favard}
The {\em Favard length} of a compact set $E\subset\rr^2$  is defined as
\begin{equation}\label{favard}
{\rm Fav}(E):=\frac1{\pi}\int_0^\pi |\proj_\theta(E)|d\theta,
\end{equation}
the average (with respect to angle) length of its linear projections.
\end{definition}

In each of the above examples, $E_\infty$ is unrectifiable, hence by a
theorem of Besicovitch we have $|\proj_\theta(E_\infty)|=0$ for almost every $\theta\in[0,\pi]$. 
It follows that ${\rm Fav}(E_n)\to 0$ as $n\to\infty$. The question is, how fast?

Examples 1 and 2 show that the pointwise rate of decay of $|\proj_\theta(E_n)|$ as $n\to\infty$, with $\theta$ fixed, depends very strongly on the choice of $\theta$. For instance, let $E_n$ is the $n$-th iteration of the 4-corner set. If $\theta=0$, we have $|\proj_0(E_n)|=2^{-n}$. However, if we let $\theta=\tan^{-1}(2)$, then $\proj_\theta(E_n)$ is the same line segment for all $n$, so that there is no decay at all (cf. Figure \ref{fig-4}). In fact, the 4-corner set has a dense set of directions where the decay is exponential in $n$, and also a dense (but measure zero, as required by Besicovitch's theorem) set of directions where $\proj_\theta(E_\infty)$ has positive measure. The same phenomenon occurs for the gasket set in Example 2. (See Kenyon \cite{Ke} and Lagarias-Wang \cite{LW1}, \cite{LW2} for more details.)

\begin{figure}[htbp]\centering\includegraphics[width=4.in]{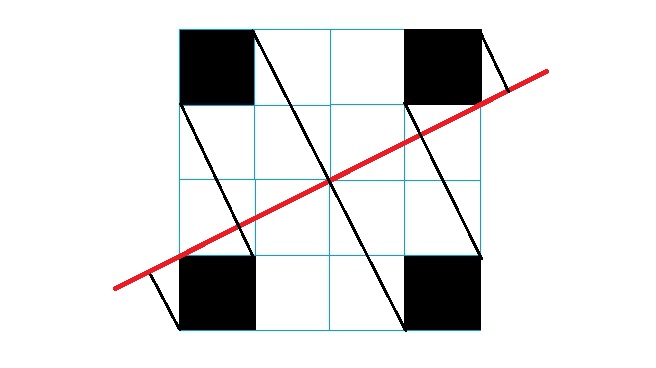}
\caption{The 4-corner set has projections of positive measure.}
\label{fig-4}
\end{figure}

The expected asymptotic behaviour of the averages Fav$(E_n)$ is neither constant nor exponentially decaying. 
A lower bound for a wide class of sets including all of the above examples is due to Mattila \cite{Mattila}: 
$${\rm Fav}(E_n)\geq \frac{C}{n}$$
Bateman and Volberg \cite{BV} improved this to $(C\log n)/n$ for the 4-corner set. Their result and proof extend to the gasket set, but seem hard to generalize beyond that. 

The known upper bounds for the sets in Examples 1-4 are summarized in the following theorem.

\begin{theorem}\label{thm-upper}
(i) We have
\begin{equation}\label{e-npv}
{\rm Fav}(E_n)\leq C n^{-p} \hbox{ for some }p>0
\end{equation}
in the following cases (with the exponent $p$ depending on the set):

\begin{itemize}
\item  the 4-corner set (Nazarov-Peres-Volberg \cite{NPV})
\item the Sierpi\'nski gasket (Bond-Volberg \cite{BV1})
\item  the general self-similar sets in Example 4 with $L=4$ (Bond-{\L}aba-Volberg \cite{BLV})
\item  the rational product sets in Example 3 under the additional ``tiling" condition that $|proj_{\theta_0}(E_\infty)|>0$ for some direction $\theta_0$ ({\L}aba-Zhai \cite{LZ}). 
\end{itemize}

\smallskip\noindent
(ii) (Bond-{\L}aba-Volberg \cite{BLV}) For rational product sets as in Example 3 with $|A|,|B|\leq 6$ (but with no additional tiling conditions), we have 
\begin{equation}\label{e-blv}
{\rm Fav}(E_n)\lesssim n^{-p/\log\log n}
\end{equation}
The cardinality assumption can be dropped under certain number-theoretic conditions on $A,B$, and in some circumstances we can improve (\ref{e-blv}) to a power bound;  see Theorem \ref{blv-thm}  below for more details.

\smallskip\noindent
(iii) (Bond-Volberg \cite{BV3}) For general self similar sets, we have the weaker bound 
\begin{equation}\label{e-bv}
{\rm Fav}(E_n)\lesssim e^{-c\sqrt{\log n}}
\end{equation}

\end{theorem}

The first general quantitative upper bound 
$Fav(E_n)\leq C\exp(-C\log^*n)$
was due to Peres and Solomyak \cite{PS1}; here,  $\log^*n$ denotes the number of iterations of the $\log$ function needed to have
$\log\dots\log n\lesssim 1$. (See also \cite{Tao} for a much weaker result in a more general setting.) The current wave of progress started with \cite{NPV}, where harmonic-analytic methods were first used in this context. The subsequent work in \cite{LZ}, \cite{BV1}, \cite{BV3}, \cite{BLV} followed the general strategy of \cite{NPV} up to a point, but also required additional new methods to deal with the increasing difficulty of the problem, especially in \cite{BLV}.

It is likely that the optimal upper bound for wide classes of self-similar sets should be ${\rm Fav}(E_n)\leq C_\epsilon  n^{-1+\epsilon}$ for all $\epsilon>0$. The example of the 4-corner set shows that the $n^\epsilon$ factors cannot be dropped. 
There are no known deterministic sets for which such an estimate is actually known, and this seems far out of reach with our current methods. The best available range of exponents in (\ref{e-npv}) at this point is $p<1/6$ (with the constant depending on $p$) for the 4-corner set \cite{NPV}.
However, Peres and Solomyak prove in \cite{PS1} that for ``random 4-corner sets"
the expected asymptotics is in fact $Cn^{-1}$. We present a simplified version of their argument here in Section \ref{PS-rewrite}.

The conditions on $A,B$ in Theorem \ref{thm-upper} (ii) are number-theoretic, and concern the roots of the polynomials
\begin{equation}\label{generating}
A(x)=\sum_{a\in A} x^a,\,\,\, B(b)=\sum_{b\in B} x^b
\end{equation}
on the unit circle. If no such roots exist, we have the power bound (\ref{e-npv}) and the restriction $|A|,|B|\leq 6$ is not necessary. Otherwise, we need more information.
Recall that for $s\in\nn$, the {\em $s$-th cyclotomic polynomial} $\Phi_s(x)$ is
\begin{equation}\label{cyclo}
\Phi_s(x):=\prod_{d:1\leq d\leq s,(d,s)=1}(x-e^{2\pi id/s}).
\end{equation}
Each $\Phi_s$ is an irreducible polynomial with integer coefficients whose roots are exactly the $s$-th roots of unity. We furthermore have the identity
$$
x^m-1=\prod_{s|m}\Phi_s(x).
$$
In particular, every $m$-th root of unity is a root of some cyclotomic polynomial $\Phi_s$ with $s|m$.

\begin{definition}\label{A1234}
We have $A(x)=\prod_{i=1}^4 A^{(i)}(x)$, where each $A^{(i)}(x)$ is a product of the irreducible factors of $A(x)$ in $\zz[x]$, defined as follows (by convention, an empty product is identically equal to 1):
\begin{itemize}
\item $A^{(1)}(x)=\prod_{s\in S_A^{(1)}}\Phi_s(x)$, $S_A^{(1)}=\{s\in\nn:\ \Phi_s(x)|A(x), (s,L)\neq 1\}$,
\item $A^{(2)}(x)=\prod_{s\in S_A^{(2)}}\Phi_s(x)$, $S_A^{(2)}=\{s\in\nn:\ \Phi_s(x)|A(x),$ $(s,L)=1\}$,
\item $A^{(3)}(x)$ is the product of those irreducible factors of $A(x)$ that have at least one root of the form $e^{2\pi i\xi_0}$, $\xi_0\in\rr\setminus\qq$ (we will refer to such roots as non-cyclotomic),
\item $A^{(4)}(x)$ has no roots on the unit circle.
\end{itemize}

\end{definition}

The factorization $B(x)=\prod_{i=1}^4 B^{(i)}(x)$ is defined similarly. 
We then have the following.

\begin{theorem}\label{blv-thm}
 (Bond-{\L}aba-Volberg \cite{BLV}) Let $E_n$ be a rational product set as in Example 3. Then the result of Theorem \ref{thm-upper} (ii) may be extended as follows:
 
 \smallskip\noindent
 (i) Assume that  $A^{(2)}\equiv B^{(2)}\equiv 1$. Then (\ref{e-blv}) holds regardless of the cardinalities of $A,B$.

\smallskip\noindent
(ii) Assume that  $|A|,|B|\leq 6$ and that $A^{(3)}\equiv B^{(3)}\equiv 1$. Then (\ref{e-blv}) can be improved to (\ref{e-npv}).

\smallskip\noindent
(iii) Assume that  $A^{(2)}\equiv B^{(2)}\equiv A^{(3)}\equiv B^{(3)}\equiv 1$. Then  (\ref{e-blv}) can be improved to (\ref{e-npv}), regardless of the cardinalities of $A,B$.

\smallskip\noindent
(iv) The condition that $|A|,|B|\leq 6$ in (ii), (iii) can be replaced by the implicit condition that
each of $A(x)$ and $B(x)$ 
satisfies the assumptions of Proposition \ref{compatible}. (If $|A|,|B|\leq 6$, then these assumptions are always satisfied.)

\end{theorem}

My goal here is to present some of the main ideas behind Theorems \ref{thm-upper} and \ref{blv-thm}, with emphasis on the number-theoretic considerations in Section \ref{sec-nt}. For the most part, I will focus on the rational product set case in Theorem \ref{thm-upper} (ii) and present the calculations in this case in some detail. The modifications needed to cover the other cases of Theorem \ref{thm-upper} will only be mentioned briefly.

The outline is as follows. In Section \ref{sec-reductions}, we set up the Fourier-analytic machinery and reduce the problem to a trigonometric polynomial estimate. This was first done in \cite{NPV}, and repeated with only minor modifications in the subsequent papers. We only sketch the arguments here, and refer the reader to \cite{BThesis} for a more detailed exposition.

In Section \ref{sec-trig}, we begin to work towards proving the main estimate. The reductions in Section \ref{p1p2} and Lemma \ref{poisson} are due to \cite{NPV}, with minor modifications in \cite{BV1}, \cite{BV3}, \cite{LZ}. The remaining issue concerns integrating a certain exponential sum on a set where another exponential sum, which we will call $|P_2(\xi)|^2$, is known to be bounded from below away from 0. 
We present two approaches to this: the SSV (Set of Small Values) method of \cite{NPV}, and the SLV (Set of Large Values) method of \cite{BLV}. 
 
In Section \ref{sec-nt}, we discuss the number-theoretic aspects of the problem for the rational product sets under consideration. The relevant facts concern the zeroes of $A(x)$ and $B(x)$ on the unit circle, and the behaviour of such zeroes under iterations of the mapping $\mathcal{L}:\ z\to z^{L}$. 
This depends very strongly on the factorization of $A$ and $B$ given in Definition \ref{A1234}. 
The factors $A^{(4)}$, $B^{(4)}$ are completely harmless and may be ignored safely.
For simplicity, we will focus on specific examples with only one type of remaining factors present.
In the case of sets as in Theorem \ref{thm-upper} (ii) and Theorem \ref{blv-thm}, the main exponential sum estimate turns out to hold for each type of factors, but for very different reasons. 

The zeroes of the ``good" cyclotomic factors $A^{(1)}$, $B^{(1)}$ are extremely well behaved under iterations of $\mathcal{L}$, in the sense that their orbits hit 1 after a finite number of iterations. It follows that the backward orbits under the multi-valued inverse mapping $\mathcal{L}^{-1}$ are well dispersed throughout the unit circle with no recurring points, a property that the SSV method relies on. This approach originated in \cite{NPV} in the special case of the 4-corner set, then was gradually expanded to its current generality in \cite{LZ} and \cite{BLV} (a similar argument appears also in \cite{BV1}).

Similar behaviour is observed for the roots of $A^{(3)}$ and $B^{(3)}$, but this is a much deeper fact related to Baker's theorem on rational approximation of logarithms of algebraic numbers; moreover, the quantitative estimates in this case are somewhat weaker, leading to the loss of $\log \log n$ in the exponent. This was done in \cite{BLV}.

The factors $A^{(2)}$, $B^{(2)}$ require a very different approach, also developed in \cite{BLV}. The roots of these factors are recurrent under iterations of $\mathcal{L}$, which leads to our exponential sum $|P_2(\xi)|^2$ having many roots of very high multiplicity. This can be handled by the SLV method, under additional number-theoretic conditions relating the structure of $S_A^{(2)}$ to the size of the set $A$. It turns out that the information we need is closely related to classical results in number theory (due to many authors) on vanishing sums of roots of unity. In Section \ref{sec-roots} we provide some of the background on this and explain the relation to our problem. 
We do not know whether this approach can suffice to prove Theorem \ref{thm-upper} (ii) for general rational product sets without the size restrictions, but we state a conjecture that, if true, would complete this part of the program.

Finally, in Section \ref{PS-rewrite} we define the random 4-corner sets of \cite{PS1} and present a proof of a Favard length estimate in this case.


\section{The Fourier-analytic approach}\label{sec-reductions}

It will be convenient to work with slopes instead of angles in the definition of Favard length (\ref{favard}). We will consider only $\theta\in[0,\pi/4]$ (the full range of $\theta$ is covered by 8 such angular segments, and our estimates will apply to all of them by symmetry), and define
$$\pi_t(x,y)=x+ty,\ t=\tan\theta.$$
Then it suffices to estimate 
\begin{equation}\label{favard-t}
\int_0^1 |    \pi_t(E_n)      |dt
\end{equation}
which is equivalent to estimating the averages of $\proj_\theta(E_n)$ up to constants.

Next, we define a ``counting function" $f_{n,t}$. Let $\mu_n$ be the 2-dimensional Lebesgue measure restricted to $E_n$, so that 
$$d\mu_n = \one_{E_n}dx$$ 
Then $L^n \mu_n$ converge weakly to the 1-dimensional Hausdorff measure $\mu_\infty$ on $E_\infty$, normalized so that $\mu_\infty(E_\infty)=1$, but we will not use this as it is the finite iterations that we are concerned with. For a given slope $0\leq t\leq 1$, consider the projected measures
$$\pi_t\mu_n(X)=\mu_n(\pi_t^{-1}(X)),\ \ X\subset \rr$$
and define $f_{n,t}$ to be the density of $\pi_t \mu_n$. Essentially, $f_{n,t}$ counts the number of squares in $E_n$ that get projected to $x$:
\begin{equation}\label{e-fn}
f_{n,t}=\sum_{a\in A_n,\, b\in B_n}\delta_{a+tb}*\chi(L^n\cdot),
\end{equation}
where 
$$A_n=\sum_{j=1}^n L^{-j}A,\ \ B_n=\sum_{j=1}^n L^{-j}B,$$
and $\chi(\cdot)$ is the density of the projection (with the given $t$) of the Lebesgue measure on the unit square. The exact form of the function $\chi$ is not important, in fact we could replace it by $\one_{[0,1]}$ and get estimates equivalent up to constants. For simplicity, we will write $f_{n,t}=f_n$ whenever displaying the dependence on $t$ is not necessary.

\begin{figure}[htbp]\centering\includegraphics[width=4.in]{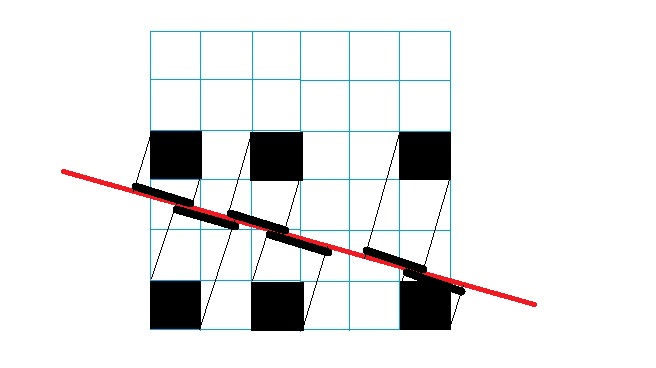}
\caption{The projected measure $\pi_t\mu_1$ for a product Cantor set}
\end{figure}

The $L^1$ norm of $f_n$ does not depend on $n$: $\|f_n\|_1=1$, uniformly in $t$. It is the higher $L^p$ norms that carry useful information. We will consider $p=2$, since this is particularly well suited to the Fourier analytic methods we wish to use. 

Heuristically, $\|f_n\|_2$ large should
corresponds to significant overlap between the projected squares. It is easy to make this argument rigorous in one direction: since $f_n$ is supported on $\pi_t(E_n)$, by H\"older's inequality 
$$1=\|f_n\|_1\leq \|f_n\|_2\cdot|\pi_t(E_n)|^{1/2}.$$
Hence if $|\pi_t(E_n)|$ is small, then $\|f_n\|_2$ must be large.

In general, H\"older's inequality only works in one direction. It is quite possible for a function with $\|f\|_1=1$ to have both large support and large $L^2$ norm. However,  in the particular case of self-similar sets considered here, there is a partial converse due to \cite{NPV}. 

\begin{proposition}\label{green-discs}
For each $t\in[0,1]$ at least one of the following holds:
\begin{itemize}
\item $\|f_{n}\|_2^2\leq K$ for all $n\leq N$,
\item $|\pi_t(E_{NK^3})| \leq C/K$.
\end{itemize}
\end{proposition}

Here $K$ is a large constant of order roughly $N^{\epsilon_0}$ (more on this shortly).
Roughly speaking, lower bounds on $\|f_n\|_2$ imply upper bounds on the size of the support of $f_N$ with $N\gg n$.
The main idea is that if ``stacking" (a large number of squares projected to the same point) occurs on some scale $L^{-n}$, then by self-similarity that phenomenon has to replicate itself throughout the set in its higher iterations, until it consumes most of the set on some much smaller scale $L^{-NK^3}$. The interested reader can find a detailed and accessible exposition in \cite{BThesis}. The ``micro-theorem converse" in Chapter 1 of \cite{BThesis} is especially recommended, as it provides a glimpse into the main idea of Proposition \ref{green-discs}, one of the essential ingredients of this approach, with only a minimum of technicalities.

We are therefore interested in lower bounds on $\|f_n\|_2$. A quick glance at Figure \ref{fig-4} again should convince the reader that there is no chance of proving non-trivial lower bounds for all $t$: if $E_n$ is the $n$-th iteration of the 4-corner set and $t=1/2$, then $\|f_n\|_2 \approx 1$ for all $n$. Similarly, $\|f_n\|_2$ can easily remain bounded in other cases where $\pi_t(E_\infty)$ has positive measure. Therefore we will instead consider the set of ``bad" directions
$$
\mathcal{E}=\{t\in[0,1]:\ \  \|f_n\|_2^2\leq K\ \ \text{for all}\ \  1\leq n\leq N\}
$$ 
for some large $K$, and attempt to prove that $|\mathcal{E}|$ is small. The choice of $K$ here depends on $A$ and $B$. If $A^{(3)}(x)=B^{(3)}(x)\equiv 1$, we will run the argument with $K:=N^{\epsilon_0}$ for some $\epsilon_0>0$; otherwise, we will have $K:=N^{\epsilon_0/\log\log N}$.

We now take Fourier transforms: since $\|f_n\|_2= \|\widehat{f_n}\|_2$, it suffices to estimate the latter from below. This is where the polynomials $A(x),B(x)$ come in. 
Define
$$\phi_A(\xi)=\frac{1}{|A|}A(e^{2\pi i\xi})=\frac{1}{|A|}\sum_{a\in A}e^{2\pi i a\xi}$$
and similarly for $B$. Let also 
$$
\phi_t(\xi):=\phi_A(\xi)\phi_B(t\xi)=\frac1{L}\sum_{(a,b)\in A\times B}e^{2\pi i(a+tb)\xi}.
$$
Then 
\begin{align*}
\widehat{f_n}(\xi)&=L^n \sum_{a\in A_n,\ b\in B_n} e^{2\pi i (a+tb)\xi} \ \widehat{\chi}(L^{-n}\xi)
\\
&=\prod_{j=0}^{n-1}\phi_t(L^{-j}\xi)\,\widehat{\chi}(L^{-n}\xi)
\end{align*}
The last term acts essentially as a cut-off function supported on $[-L^{n},L^n]$. 

We omit the calculations and pigeonholing steps required at this point. We also pass from the set of bad directions $\mathcal{E}$ to a certain subset of it of proportional size; for simplicity, we will continue to denote this set by $\mathcal{E}$ here. Finally, we rescale the resulting Fourier transforms so that the relevant integrals live on the interval [$0,1]$. The conclusions are summarized in the next proposition.

\begin{proposition}\label{prop-reduction}
To prove Theorem \ref{thm-upper}, we need only prove the following:
Let $\epsilon_0>0$ be sufficiently small, $\frac{N}{4}\leq n\leq \frac{N}{2}$, and assume that 
\begin{equation}\label{size-e}
|\mathcal{E}|\geq K^{-1/2}
\end{equation}
Then there is a
$t\in\mathcal{E}$ such that 
\begin{equation}\label{intro-e1}
\int_{L^{-m}}^1 \prod_{j=1}^{n} |\phi_t(L^{j}\xi)|^2 d\xi
\geq cKL^{-n}N^{-\alpha\epsilon_0}
\end{equation}
for some constants $c,\alpha>0$, depending on $A$ and $B$ but not on $\epsilon_0$.
\end{proposition}

Here $m$ is much smaller than $n$: if $A^{(3)}\equiv B^{(3)}\equiv 1$, we let $m= c_0\log N$ (rounded to an integer), otherwise $m=c_0\frac{\log N}{\log\log N}$.

We have not discussed the case of general self-similar sets as in Theorem \ref{thm-upper} (iii), where the similarity centers do not necessarily form a rational product set. For such sets, the same general outline is followed so far and Proposition \ref{prop-reduction} still applies. We still have a counting function $f_{n,t}$ and a trigonometric polynomial $\phi_t(\xi)$, although $\phi_t$ need not factorize as above and the powers of $e^{2\pi i\xi}$ need not be integer. We set $K\approx \exp(\epsilon_0\sqrt{\log N})$ and $m\approx \epsilon_0\sqrt{\log N}$, which at the end of the day yields the bound $\exp(-c\sqrt{\log n})$ on the Favard length of $E_n$.

\bigskip


\section{Trigonometric polynomial estimates}
\label{sec-trig}

\subsection{The separation of frequencies}
\label{p1p2}

Before we proceed further, it is important to understand a major issue arising in (\ref{intro-e1}). At first sight, there might seem to be little difference between the integral in (\ref{intro-e1}) and a similar integral taken from 0 to 1. These two integrals, however, may actually behave very differently, depending on the trigonometric polynomial being integrated.

As a warm-up, we will try to estimate $\int_0^1 |P(\xi)^2 d\xi$, where
$$
P(\xi)=\prod_{j=1}^n \phi_t(L^j\xi).
$$
The following argument is based on an idea of Salem. We write $P(\xi)$ as a long trigonometric polynomial $P(\xi)=L^{-n}\sum_{\alpha\in \mathcal{A}}
e^{2\pi i\alpha\xi}$, where $\mathcal{A}\subset\rr$ and $|\mathcal{A}|=L^{n}$. (Note that $\alpha$ need not be integer, otherwise we could simply evaluate the integral directly.)
We have $P(\xi)=\overline{P(-\xi)}$, so that $\int_0^1|P|^2=\frac{1}{2}\int_{-1}^1|P|^2$.
Let $h(x)=\one_{[0,1/2]}*\one_{[-1/2,0]}$, then $0\leq h\leq C$, $\supp h\subset [-1,1]$ and $\widehat{h}\geq 0$. Therefore
\begin{equation}\label{intro-e2a}
\begin{split}
\int_0^1|P(\xi)|^2d\xi &\geq C^{-1}\int_{-1}^1 |P(\xi)|^2h(\xi)d\xi\\
&=C^{-1}L^{-2n} \sum_{\alpha,\alpha'}\int h(\xi)e^{2\pi i(\alpha-\alpha')\xi}d\xi\\
&=C^{-1}L^{-2n} \sum_{\alpha,\alpha'}\widehat{h}(\alpha-\alpha')\\
&\geq C^{-1}L^{-2n} \sum_{\alpha=\alpha'}\widehat{h}(0)\\
&\geq C^{-1} L^{-2n}|\mathcal{A}|=C^{-1} L^{-n}.\\
\end{split}
\end{equation}
This would be a perfectly good bound for us to use, if we could get it for the slightly smaller interval in (\ref{intro-e1}) instead. There, however, lies the crux of the matter. Many trigonometric polynomials peak out at 0, then become very small outside a neighbourhood of it. (This would for example happen with high probability for $P(\xi)=L^{-n}\sum_{\alpha\in \mathcal{A}}
e^{2\pi i\alpha\xi}$ if the exponents $\alpha$ were chosen at random from some large interval, instead of being given by a self-similar set.) The inequality (\ref{intro-e2a}) is therefore of little use to us, given that the main contribution to the integral is expected to come from the peak at 0, and we are trying to bound from below a part of the integral that could well be much smaller by comparison.

The reader might ask at this point why we need to integrate on $[L^{-m},1]$ instead of $[0,1]$ in the first place. This comes from the pigeonholing arguments that we skipped in the previous section. Roughly, if $P(\xi)$ were indeed too close to a rescaled Dirac delta function at 0, then $f_n$ (as its rescaled inverse Fourier transform) would be close to a constant function, as in Figure \ref{fig-4}. This is exactly the type of behaviour that we are trying to disprove, or at least confine to a small set of projection angles. Instead, we are looking for irregularities of distribution of $f_n$, associated with somewhat large values of $P(\xi)$ away from 0.

If we separate $P(\xi)$ into low and high frequencies, it turns out that we have much better control of the high frequency part of $P$. We will write $P=P_1P_2$, where 
$$
P_1(\xi)=\prod_{j=m+1}^n \phi_t(L^j\xi),\ \ 
P_2(\xi)=\prod_{j=1}^m \phi_t(L^j\xi),
$$
It is immediate to verify that the argument in (\ref{intro-e2a}) also yields
\begin{equation}\label{intro-e2}
 \int_0^1|P_1(\xi)|^2d\xi \geq C^{-1} L^{m-n}.
\end{equation}
Crucially, now we also have some control of what happens on $[L^{-m},1]$. 

\begin{lemma}\label{poisson-lemma}
For $t\in \mathcal{E}$, we have
\begin{equation}\label{poisson}
I_0:= \int_0^{L^{-m}}|P_1|^2\leq C_0KL^{-n}.
\end{equation}
\end{lemma}

Assuming that $K<cL^m$ for some small enough constant, we now have 
$$\int_{L^{-m}}^1 |P_1|^2 \geq (2C)^{-1} L^{m-n}$$
That would still do us little good if most of this integral lived on the part of $[L^{-m},1]$ where $P_2$ is very small. We must therefore prove that this is not the case.

Currently, there are two ways of doing this: the SSV method used in \cite{NPV}, \cite{LZ},  \cite{BV1}, \cite{BV3}, \cite{BLV}, and the SLV method of \cite{BLV}. We present them in Sections \ref{sec-ssv} and \ref{sec-slv}, respectively. 
In the proof of Theorem \ref{thm-upper} (ii) in \cite{BLV}, the two methods are combined together so that the SSV method is applied to the factors $A^{(i)}, B^{(i)}$ with $i=1,3$ (we will call them {\it SSV factors}), and the SLV method handles the factors $A^{(2)},B^{(2)}$ ({\it SLV factors}). 
This is done by proving an SLV bound first and then subtracting the SSV intervals from the SLV set.
In order to keep the exposition as simple as possible, I will present each of the two methods separately, assuming that only one type of factors is present at a time.


\subsection{SSV estimates}\label{sec-ssv}

The results in \cite{NPV}, \cite{BV1}, \cite{BV3}, \cite{LZ}, and parts of \cite{BLV}, rely on estimates on the size of the {\it Set of Small Values} (SSV) of $P_2$, which we now define. For a function $\psi:\nn\to(0,\infty)$, with $\psi(m)\searrow 0$ as $m\to\infty$, we write
\begin{equation}\label{ssvpropdef}
SSV_t =\{\xi\in[0,1]:\ | P_2(\xi)|  \leq\psi(m)\}
\end{equation}
Suppose that we can prove that for some $t\in \mathcal{E}$,
\begin{equation}\label{ssv-e2}
\int_{[L^{-m},1]\cap SSV_t} |P_1(\xi)|^2 d\xi \leq \frac{C^{-1}}{2}L^{m-n}.
\end{equation}
Then by (\ref{intro-e2}), the integral on $[0,L^{-m}]\cup ([L^{-m},1]\setminus SSV_t)$ is at least $\frac{C^{-1}}{2}L^{m-n}$.  Assume also that this dominates (\ref{poisson}) as explained in the last section. 
Since on the complement of $SSV_t$ we have $|P_2|\geq \psi(m)$, we get that
\begin{equation}\label{ssv-e3}
\int_{[L^{-m},1]\setminus SSV_t} |P(\xi)|^2 d\xi \geq cL^{m-n}\psi(m)^2,
\end{equation}
which of course also provides a bound from below on $\int_{L^{-m}}^1 |P(\xi)|^2d\xi$.

In order for us to be able to apply Proposition \ref{prop-reduction}, this bound must be at least as good as (\ref{intro-e1}). This depends on the choices of $\psi$ and $K$. Our method of proving estimates of the form (\ref{ssv-e2}) will depend on the {\it SSV property} defined below, and $\psi$ must be chosen so as for $\phi_t$ to have this property. This puts limits on how large the parameter $K$ is allowed to be, and at the end of the day, determines the type of Favard length bounds that we are able to get.

\begin{definition}\label{SSV prop}
We say that $\phi_t$ has the SSV property with SSV function $\psi$ if there exist $c_2,c_3>0$ with $c_3/ c_2$ sufficiently large (to be determined later) such that $SSV_t$ is contained in $L^{c_2m}$ intervals of size $L^{-c_3m}$. 
\end{definition}

The function $\psi$ depends on $A$ and $B$. If $A^{(3)}\equiv B^{(3)}\equiv 1$, we will have $\psi(m)=L^{-c_1m}$; otherwise, we set $\psi(m)=L^{-c_1m\log m}$. This matches our choices of 
$K:=N^{\epsilon_0}$ and $K:=N^{\epsilon_0/\log\log N}$, respectively. 

For periodic trigonometric polynomials such as $\phi_A$, the SSV property can be thought of as a condition on the separation of the roots of $P_1$. Generically, if $P_1$ has at most $L^{c_2m}$ roots, all of multiplicity bounded uniformly in $m$ and roughly equally spaced, the SSV property holds with $c_3$ arbitrarily large provided that $c_1$ was chosen to match it. (The values of $P_1$ become smaller as we zoom in closer to the roots, but the number of intervals needed for this stays constant.) On the other hand, high multiplicity roots of $P_1$ can lead to SSV violations.

The argument we present here was first used in \cite{NPV} for the 4-corner set.
Its subsequent applications in \cite{LZ}, \cite{BV1}, \cite{BLV} relied on much more general number-theoretic input, in the sense that the SSV property was extended to wider classes of sets, but the calculation in Proposition \ref{ssv goal} continued to be used with relatively few substantive changes.

Our proof of Proposition \ref{ssv goal} relies strongly on the fact that $E_n$ is a product set, which will allow us to separate variables in the double integrals below. 
For non-product self-similar sets (Examples 2 and 4) with $L=3$ and 4, there is a ``pseudofactorization" substitute for this, where the trigonometric polynomials in question do not actually factor into functions of a single variable, but can nonetheless be estimated from below by combinations of such functions. This was done in \cite{BV1} for $L=3$ and in \cite{BLV} for $L=4$. 

For non-product sets with $L\geq 5$, the SSV property is still used, but the argument below must be replaced by an entirely different one, based on a variant of the Poisson localization lemma (Lemma \ref{poisson}). Moreover, we can only get a weaker SSV property with $\psi(m)=L^{-c_1m^2}$ and must choose 
$K=\exp(\epsilon_0 \sqrt{\log N})$ to match that. This was done in \cite{BV3}, and yields the result in Theorem \ref{thm-upper}(iii).

We now prove the result we need for product sets. 
Let 
$$SSV_A:=\{\xi\in[0,1]:|P_{2,A}(\xi)|\leq \psi(m)\},$$
$$SSV_B(t):=\{\xi\in[0,1]:|P_{2,B}(t\xi)|\leq\psi(m)\}$$
where $P_{2,A}=\prod_{j=1}^m \phi_A(L^j \xi)$ and similarly for $B$. Then
$SSV_t\subseteq SSV_A\cup SSV_{B}(t)$. The SSV properties for $\phi_A$ and $\phi_B$ are defined in the obvious way.

\begin{proposition}\label{ssv goal}
Suppose that $\epsilon_0$ is small enough, and that (\ref{size-e}) holds.
Assume also that $\phi_A,\phi_B$ have the SSV property.
Then
$$I:=\frac1{|\mathcal{E}|}\int_0^1\int_{[L^{-m},1]\cap SSV_t}|P_{1,t}(\xi)|^2 d\xi dt\leq CL^{m-n}$$
Since $|\mathcal{E}|\leq 1$, this in particular implies (\ref{ssv-e2}).
\end{proposition}

\begin{proof}
It suffices to estimate
$$I_A:= \frac1{|\mathcal{E}|}\int_0^1\int_{[L^{-m},1]\cap SSV_A}|P_{1,A}(\xi) P_{1,B}(t\xi)|^2 d\xi dt$$
the proof for the integral on $SSV_B(t)$ being similar.

We will need to split $P_1$ further into frequency ranges, and we set up the notation for this:
$$A_{m_1}^{m_2}(x)=\prod_{k=m_1+1}^{m_2}A(x^{L^k}),$$
and similarly for $B$. (Note that this is not normalized, so that $P_{1,A}(\xi)=|A|^{m-n}A_m^n(e^{2\pi i\xi})$.) The reason for this is that high-frequency factors $A_\ell^n$, with $\ell>m$ sufficiently large depending on the constants in the SSV estimates, are well behaved on the SSV intervals of $\phi_A$. 

The following lemma is very simple, but we single it out because it will ultimately provide the gain we seek.

\begin{lemma}\label{exp dichot}
We have
$$\int_{\xi_0}^{\xi_0+L^{-m_1}} |A_{m_1}^{m_2}(e^{2\pi i\xi})|^2\, d\xi =  |A|^{m_2-m_1}L^{-m_1}$$

\end{lemma}

\begin{proof}

Expanding $|A_{m_1}^{m_2}|^2$, we get
$$\int_{\xi_0}^{\xi_0+L^{-m_1}} |A_{m_1}^{m_2}(e^{2\pi i\xi})|^2 d\xi
=\int_{\xi_0}^{\xi_0+L^{-m_1}}
\sum_{j_1,j_2=1}^{|A|^{m_2-m_1}}e^{2\pi iL^{m_1}(\lambda_{j_1}-\lambda_{j_2})\xi}\, d\xi
$$ 
where $\lambda_j\in \nn$ are distinct. All terms with $j_1\neq j_2$ integrate to 0 by periodicity, and each diagonal term contributes  $L^{-m_1}$.

\end{proof}

We now return to the proof of Proposition \ref{ssv goal}. We may assume that $c_2,c_3\geq 2$. Let also $\ell =c_3m$, and let $J_i$, $i=1,\dots,M$, be the SSV intervals for $\phi_A$ that intersect $[L^{-m},1]$. Then $|J_i|=L^{-\ell}$ and $M\leq L^{c_2m}$. 

We begin by changing variables $(\xi,t)\to (\xi,u)$, where $u=\xi t$, $dt=du/\xi$. 
Then
$$
I_A\leq\frac1{|\mathcal{E}|}L^{-2(n-m)}\sum_{j=1}^{M}\int_{J_i}
|A_m^n(e^{2\pi i\xi})|^2     \      \frac{d\xi}{\xi}\int_0^1 |B_m^n(e^{2\pi iu})|^2 du$$
By Lemma \ref{exp dichot},
\begin{align*}
I_A&\leq\frac{|B|^{n-m}}{|\mathcal{E}|}L^{-2(n-m)}
\sum_{j=1}^{M}\int_{J_i}   |A_m^n(e^{2\pi i\xi})|^2\frac{d\xi}{\xi}\\
&\leq\frac{2 |B|^{n-m}}{|\mathcal{E}|}L^{-2(n-m)}L^m
\sum_{j=1}^{M}\int_{J_i}   |A_m^n(e^{2\pi i\xi})|^2\, d\xi
\end{align*}
where we also used that $|\xi|\geq L^{-m}/2$ on $J_i$. We now separate frequencies in order to apply Lemma \ref{exp dichot} to integration on $J_i$: 
\begin{align*}
\int_{J_i}   |A_m^n(e^{2\pi i\xi})|^2\, d\xi
&=\int_{J_i} |A_m^\ell(e^{2\pi i\xi})A_\ell^n(e^{2\pi i\xi})|^2 d\xi\\
&\leq  |A|^{2(\ell -m)}  \int_{J_i}   |A_\ell^n(e^{2\pi i\xi})|^2 d\xi\\
&\leq |A|^{2(\ell -m)+(n-\ell)}L^{-\ell}  
\end{align*}
Hence
\begin{align*}
I_A&\leq  \frac{2 |B|^{n-m}}{|\mathcal{E}|}L^{-2(n-m)}L^{m-\ell}M\,|A|^{n+\ell -2m}
=\frac{2}{|\mathcal{E}|}\frac{L^{(c_2+1)m}}{|B|^{(c_3-1)m}}L^{-n}. \\
\end{align*}
If $c_3$ is chosen large enough so that $|B|^{c_3-1}\geq L^{c_2+1}$, and if (\ref{size-e}) holds, the last expression is at most 
$$2|\mathcal{E}|^{-1}L^{-n}\leq 4 K^{1/2}L^{-n}< cL^{m-n}$$
as required.


\end{proof}


\subsection{Salem's argument on difference sets}
\label{sec-slv}

The SSV argument in the last section is quite general and suffices to prove a power bound for many self-similar sets including the 4-corner set and the Sierpi\'nski gasket in Examples 1 and 2. Unfortunately, it is not strong enough quantitatively to yield a power bound for all self-similar sets, or even for all rational product sets of the type discussed here. In this generality, we only have an SSV estimate with $\psi(m)=L^{-cm^2}$, leading to the Favard length bound in Theorem \ref{thm-upper} (iii) \cite{BV3}. It was observed in \cite{BLV} that this is in fact the best SSV estimate that we can have in the general case, basically because $P_1$ may have very high multiplicity zeroes and therefore take very small values in relatively large neighbourhoods of such zeroes. (The number-theoretic reasons for such behaviour will be explained in more detail in Section \ref{sec-badcyclo}.)

The SLV approach, introduced in \cite{BLV}, circumvents this difficulty as follows.  
Although $P_2$ may be unacceptably small on intervals too long to be negligible for our purposes, it is still reasonably large on most of $[0,1]$. This raises the prospect of reworking the calculation in (\ref{intro-e2a}) so that the integration only takes place on the set where $P_2$ is not small and no SSV intervals need to be subtracted afterwards. 

This is indeed what we do in \cite{BLV}, with one major caveat. The argument in (\ref{intro-e2a}) relies on the availability of a function $h(\xi)$ supported on $[-1,1]$ whose Fourier transform is non-negative. In general, if we replace $[-1,1]$ by a subset $G$ of it, the existence of such a function can no longer be taken for granted. However, if $G$ contains a {\it difference set} 
$$\Gamma -\Gamma:=\{x-y:\ x,y\in\Gamma\}$$
for an appropriate $\Gamma$ (in our application, a finite union of intervals), then such a function can in fact be constructed, and we will do so below.

The challenge is to get a good enough bound from below on the resulting integral. Specifically, we  must have
\begin{equation}\label{good est}
\int_{G}|P_1|^2\geq CKL^{-n}
\end{equation}
with $C>C_0$, where $C_0$ is the constant in Lemma \ref{poisson}. This will allow us to remove the interval $[0,L^{-m}]$ from $G$ and still get the lower bound in (\ref{intro-e1}). Given that (\ref{poisson}) is essentially optimal, the requirement (\ref{good est}) cannot be relaxed.

This leads to competing demands on $G$, therefore on $\Gamma$: on one hand, the difference set $\Gamma-\Gamma$ has to avoid the high multiplicity zeroes of $P_1$, and on the other hand, it also must be large enough for (\ref{good est}) to hold. A major challenge in this approach is to make sure that the two conditions can be satisfied simultaneously.

Our use of difference sets in this context was inspired by similar calculations on Bohr sets associated with exponential sums in additive combinatorics  (see e.g. \cite{Bourg1999}). Unfortunately, the usual additive-combinatorial lower bounds on the size of such sets are not sufficient for our purposes, as they fail to ensure (\ref{good est}). Instead, our set $\Gamma$ will be tailored to the problem, based on specific number-theoretic information regarding the cyclotomic divisors of $A(x)$ and $B(x)$. We defer a discussion of this issue until the next section, focusing for now on obtaining the lower bound in (\ref{good est}) if $\Gamma$ is given.

\begin{definition}\label{slv structured}
We say that $\phi_t$ is SLV-structured if there is a Borel set $\Gamma\subset[0,1]$ (the SLV set) and constants $C_1,C_2$ such that:
\begin{equation}\label{e-add1a}
\prod_{k=1}^m |\phi_t(L^k\xi)|\geq L^{-C_1 m}\hbox{ on }\Gamma-\Gamma,
\end{equation}
\begin{equation}\label{e-add2a}
|\Gamma|\geq C_2KL^{-m}.
\end{equation}
\end{definition}

\begin{proposition}\label{gamma want}
Suppose that $\phi_{t_0}(\xi)$ is SLV-structured. Then \eqref{intro-e1} holds for $t=t_0$.
\end{proposition}

\begin{proof}
Similarly to (\ref{intro-e2a}), we write $P_1(\xi)=\sum_{\alpha\in\mathcal{A}}
e^{2\pi i\alpha\xi}$ (note that $|\mathcal{A}|=L^{n-m}$). Observe that $|P_i(\xi)|=|P_i(-\xi)|$, $i=1,2$, so that the integrands below are symmetric about the origin.

Let $h=|\Gamma|^{-1}\one_\Gamma*\one_{-\Gamma}$, then $0\leq h \leq 1$ and 
$\widehat{h}=|\Gamma|^{-1}|\widehat{\one_\Gamma}|^2\geq 0$. Hence
\begin{align*}
\int_{\Gamma-\Gamma}|P_1(\xi)|^2&\geq \int_{\Gamma-\Gamma}|P_1(\xi)|^2h(\xi)d\xi\\
&\geq C L^{-2(n-m)}\sum_{\alpha,\alpha'}\int_{\Gamma-\Gamma}h(\xi)e^{2\pi i(\alpha-\alpha')\xi}d\xi\\
&\geq C L^{-2(n-m)}\Big(\sum_{\alpha}\int_{\Gamma-\Gamma}h(\xi)d\xi
+\sum_{\alpha\neq \alpha'}\widehat{h}(\alpha-\alpha')\Big)\\
&\geq C L^{-2(n-m)}L^{n-m}|\Gamma|=CL^{m-n}|\Gamma|\\
& \geq C C_2KL^{-n},
\end{align*}
where at the last step we used (\ref{e-add2a}). 
Comparing this with (\ref{poisson}), we get
$$
\int_{(\Gamma-\Gamma)\setminus[-L^{-m},L^{-m}]}|P_1(\xi)|^2d\xi \geq C_0KL^{-n},
$$
hence using also (\ref{e-add1a}),
$$
\int_{L^{-m}}^1 |P_1(\xi)|^2\,|P_2(\xi)|^2\,d\xi
\gtrsim C_0KL^{-n}L^{-2C_1m} 
\gtrsim KL^{-n}N^{-\alpha\epsilon_0}
$$
for some $\alpha>0$. The last inequality holds by the choice of $m$.

\end{proof}


\section{The number-theoretic part}
\label{sec-nt}

We now show how the SSV and SLV properties required to prove Theorem \ref{thm-upper} (ii) for rational product sets follow from number-theoretic properties of $A(x)$ and $B(x)$. Throughout this section we will be referring to the factorization in Definition \ref{A1234}. For simplicity, we will focus on product sets where only one type of SSV or SLV factors is present at a time: SSV cyclotomic factors in Section \ref{sec-goodcyclo}, SSV non-cyclotomic factors in Section \ref{sec-noncyclo}, and SLV factors in Section \ref{sec-badcyclo}. This provides a good introduction to the main ideas while minimizing the technicalities. Generalizing the construction in Section \ref{sec-badcyclo} depends on a more detailed analysis of the possible SLV cyclotomic divisors of $A(x)$. We discuss this in Section \ref{sec-roots}.

\subsection{Telescoping products}\label{sec-goodcyclo}

We begin with the case when $A(x)$ and $B(x)$ only have SSV factors of the first type.

\begin{proposition}\label{goodcyc}
Suppose that the only roots of $A(x),B(x)$ on the unit circle are roots of cyclotomic polynomials $\Phi_s(x)$ with $(s,L)\neq 1$. Then $\phi_t(\xi)$ has the SSV property with constants uniform in $t$.
\end{proposition}

Proposition \ref{goodcyc} was first proved (in a somewhat camouflaged form) in \cite{NPV} for the special case of the 4-corner set. The subsequent papers \cite{LZ}, \cite{BLV} made it more explicit and extended it to all factors $A^{(1)}(x)$, $B^{(1)}(x)$ as in Definition \ref{A1234}. 
For simplicity of exposition, we only prove the proposition for the 4-corner set, then discuss the general case very briefly.

\begin{proof} 
We have $A=B=\{0,3\}$ and $L=4$. Then 
$$A(x)=1+x^3=(1+x)(1-x+x^2)=\Phi_2(x)\Phi_6(x)$$
and neither $2$ nor $6$ are relatively prime to 4, so that this is indeed a special case of Proposition \ref{goodcyc}.
We will take advantage of the identity
\begin{equation}\label{e-tele1}
A(x)B(x^2)=(1+x^3)(1+x^6)=1+x^3+x^6+x^9=\frac{x^{12}-1}{x^3-1}.
\end{equation}
Iterating (\ref{e-tele1}), we can express the long products appearing in $P_{2}$ in a closed form:
\begin{equation}\label{e-tele2}
\prod_{j=1}^{m} A(x^{4^j})B(x^{2\cdot 4^j})
=\prod_{j=1}^{m}\frac{x^{3\cdot 4^{j+1}}-1}{x^{3\cdot 4^j}-1}
=\frac{x^{3\cdot 4^{m+1}}-1}{x^{12}-1}
\end{equation}
In particular, the zeroes of 
$\prod_{j=1}^m A(e^{2\pi i4^j\xi })$ all have multiplicity 1 and are contained in the arithmetic progression $(3\cdot 4^m)^{-1}\zz$. 
This controls the sets of small values of $P_{2,A}$ and $P_{2,B}$ simultaneously:
$$
|P_{2,A}(\xi)|=2^{-m}\prod_{j=1}^{m} |A(x^{4^j})| 
= 2^{-m}\ \frac{ \prod_{j=1}^{m} |A(x^{4^j})B(x^{2\cdot 4^j})| }{\prod_{j=1}^{m} |B(x^{2\cdot 4^j})|}
$$
$$
\geq \frac{1}{2\cdot 4^m}\left| e^{2\pi i 3\cdot 4^{m+1}\xi}-1            \right|
$$
so that $\phi_A$ has the SSV property, with 
$$
\{\xi:\ |P_{2,A}(\xi)|\leq 4^{-m} \delta \}\subset \bigcup_{a\in\zz} \left( \frac{a-c\delta}{3\cdot 4^{m+1}}, 
\ \frac{a+c\delta}{3\cdot 4^{m+1}} \right)
$$
Of course, a similar argument applies to $\phi_B$.

\end{proof}

The identity (\ref{e-tele1}) has a natural geometric interpretation: the projection of the 4-corner set on a line with slope 2 (the exponent of $x$ in $B(x^2)$) is a line segment. 
This was used in \cite{LZ}, where an analogous identity in more general cases was deduced from the ``tiling" condition that $|proj_{\theta_0}(E_\infty)|>0$ for some direction $\theta_0$. The argument was extended to its present generality in \cite{BLV}. In this setting, there need not be a single identity such as (\ref{e-tele1}) that covers all zeroes of $P_{2,A}$ and $P_{2,B}$ simultaneously, but we can still get the SSV estimate from similar telescoping products for individual cyclotomic factors of $A$ and $B$.


\subsection{Non-cyclotomic roots}\label{sec-noncyclo}

Consider now the case when $A(x)$, $B(x)$ have roots on the unit circle that are not roots of unity. To see that this indeed may happen, let  $A=B=\{0,3,4,5,8\}$ and $L=25$. 
Then $A(x)=1+x^3+x^4+x^5+x^8$ has 4 roots on the unit circle, all of which are non-cyclotomic.
(Namely, the roots are $e^{2\pi i\omega_j}$ with $\omega_1\approx 0.316$, $\omega_2
\approx 0.457$, $\omega_3\approx 0.543$, $\omega_4\approx 0.684$).

The argument below is due to \cite{BLV}, and relies 
on a version of Baker's theorem in transcendental number theory.
Roughly speaking, the theorem states that if $e^{2\pi i\xi_0}$  is a root of $A(x)$ with $\xi_0\in\rr\setminus\qq$, then $\xi_0$ cannot be approximated too well by rational numbers. 
The precise statement we need is a corollary of Theorem 9.1 of \cite{wald}.

\begin{theorem}\label{baker}
If $\xi_0\in\rr$ is irrational and $z_0=e^{2\pi i\xi_0}$ is algebraic, then for any integers $a,q$ with $q> 0$ we have
\begin{equation}\label{pseudoliouville}
\Big|\xi_0-\frac{a}{q}\Big|\geq \frac{C_0}{q^{\alpha}},
\end{equation}
where $C_0>0,$ $\alpha>1$ are positive constants that may depend on $\xi_0$, but are independent of
$a,q$. 
\end{theorem}

Based on this, we can prove the following.

\begin{proposition}\label{liouville}
Suppose that all roots of $A(x)$ on the unit circle are non-cyclotomic. Then $A(x)$ has the log-SSV property.
\end{proposition}

To prove Proposition \ref{liouville}, it suffices to consider separately each factor $\varphi(\xi)=e^{2\pi i\xi}-e^{2\pi i\xi_0}$, where $\xi_0\in[0,1]\setminus\qq$ and $A(e^{2\pi i\xi_0})=0$.
Note that for $j=0,1,\dots,m-1$, the zeroes of $\varphi(L^j\xi)$ are $L^{-j}(\xi_0+k)$, $k\in\zz$. We wish to prove that these zeroes do not accumulate too closely.

The main observation is the following.
Let $j'>j$, then the distance between any two roots of $\varphi(L^j\xi)$ and $\varphi(L^{j'}\xi)$ respectively is
\begin{equation}\label{e-irr1}
\left| \frac{\xi_0+k}{L^j}-\frac{\xi_0+k'}{L^{j'}}\right|
=\frac{L^{j'-j}-1}{L^{j'}}\left|\xi_0+\frac{a}{L^{j'-j}-1}\right|
\end{equation}
for some $a\in\zz$. By (\ref{pseudoliouville}),
$$
(\ref{e-irr1})\geq\frac{L^{j'-j}-1}{L^{j'}}\,\,\frac{C_0}{(L^{j'-j}-1)^\alpha}
\geq C_0L^{-j}L^{-(j'-j)\alpha}.
$$
A counting argument converts this into the log-SSV property; see \cite{BLV} for details.

It is this part of the argument that causes the loss of $\log\log n$ in the exponent in (\ref{e-blv}). If $A(x)$ and $B(x)$ have no roots on the unit circle that are not roots of unity, then this section may be omitted altogether, and our methods yield the stronger power bound on the Favard length of the set. 

Naturally, the question arises whether the current argument could be improved to yield the stronger SSV property and hence the Favard power bound. It is not difficult to check that this cannot be accomplished by using so-called effective versions of Theorem \ref{baker}. Any diophantine result of the form (\ref{pseudoliouville}), regardless of the values of $C_0$ and $\alpha$, still yields only the log-SSV property. On the other hand, Baker's theorem is very general, and it could be possible to develop a better argument based on information specific to the problem. For instance, our proof does not really invoke approximating $\xi_0$ by arbitrary rationals, but only by those in (\ref{e-irr1}).


\subsection{The construction of $\Gamma$}
\label{sec-badcyclo}

We now turn to SLV factors, starting with an instructive special case.
Let  $A=B=\{0,3,4,8,9\}$ and $L=25$. Then $A(x)=B(x)=1+x^3+x^4+x^8+x^9$, which is divisible by
$\Phi_{12}(x)=1-x^2+x^4$. We claim that the long products
$$\prod_{j=1}^m A(x^{25^j}),\ \ \prod_{j=1}^m B(x^{25^j})$$
have zeroes of very high multiplicity on the unit circle. 
Indeed, let $z_0$ be a root of $\Phi_{12}$, say $z_0=e^{\pi i/6}$. Since 12 is relatively prime to $25$, the numbers  $e^{25^j \pi i /6}$ for $j=1,2,\dots$ are again roots of $\Phi_{12}$, hence also roots of $A(x)$. It follows that $z_0$ is a root of $A(x^{25^j})$ for all $j=1,\dots,m$, hence a root of the long product with multiplicity $m$. Similarly, roots of $\Phi_{12}(x^{25^k})$ are roots of $A(x^{25^j})$ for $j=k,\dots,m$, hence they are also high multiplicity roots of the long product as long as $k$ is reasonably small compared to $m$. This yields a large number of high multiplicity zeroes that cannot be handled by SSV methods. (The same argument applies of course to $B$.)

In cases such as this, we need to construct $\Gamma$ as specified in Section \ref{sec-slv}. More precisely, we will construct sets $\Gamma\subset[0,1]$ and $\Delta\subset\rr$ such that
\begin{equation}\label{e-add1}
\Gamma-\Gamma\subset\Delta
\end{equation}
\begin{equation}\label{e-add2}
\prod_{j=0}^{m-1}|\phi_A(25^j\xi)\phi_B(25^j t\xi)|^2\geq 25^{-C_1m}\hbox{ for }\xi\in\Delta
\end{equation}
\begin{equation}\label{e-add3}
|\Gamma|\geq C_2\,25^{-(1-\epsilon)m}\hbox{ for some }\epsilon>0.
\end{equation}

The interested reader may check that $A(x)$ and $B(x)$ have no roots on the unit circle other than the roots of $\Phi_{12}(x)$, hence the construction will resolve the problem entirely in this particular case, yielding a power bound on the Favard length of $E_n$.

We first construct a set $\Delta_0$ disjoint from the set of small values of $\phi_{A}(\xi)=\frac{1}{5}A(e^{2\pi i \xi})$. 
Let $\Lambda=\{\frac{1}{12}, \frac{5}{12}, \frac{7}{12}, \frac{11}{12}\}+\zz$, so that $e^{2\pi i\lambda}$ for $\lambda\in\Lambda$ are exactly the zeroes of $\Phi_{12}$. We want $\Delta_0$ to avoid a neighbourhood of $\Lambda$. The key observation is that all points of $\frac{1}{6}\zz$ are at distance at least $1/12$ from $\Lambda$, hence we may take $\Delta_0$ to be a neighbourhood of $\frac{1}{6}\zz$. We are using here that 6 divides 12, but $\phi_0$ does not vanish at any 6-th root of unity; this is the property that we will try to generalize in the next subsection.

We now turn to the details. Let 
$$\Delta_0=\frac1{6}\zz+\Big(-\frac{\eta}{12},\frac{\eta}{12}\Big)$$
for some $\eta\in(0,1)$. Then there is a constant $c=c(\eta)>0$ such that 
$$\phi_0(\xi)\geq c\hbox{ for }\xi\in\Delta_0.$$
Crucially, since $\Delta_0$ was defined as a neighbourhood of the additive group $\frac{1}{6}\zz$, it can be expressed as a difference set. For example, we have
$$\Delta_0=\Gamma_0-\Gamma_0,\ \ 
\Gamma_0=\frac1{6}\zz+\Big(0,\frac{\eta}{12}\Big).$$

This will be our basis for the construction of $\Gamma$. 
By scaling, we have
$$\phi_A(25^j\xi)\geq c\ \hbox{ for }\ \xi\in\Delta_j:= \frac{25^{-j}}{6}\zz+\Big(-\frac{25^{-j}\eta}{12},\frac{25^{-j}\eta}{12}\Big),
$$
and similarly,
$$\phi_B(25^jt\xi)\geq c\ \hbox{ for }\ \xi\in t^{-1}\Delta_j:= \frac{25^{-j}}{6t}\zz+\Big(-\frac{25^{-j}\eta}{12t},\frac{25^{-j}\eta}{12t}\Big).
$$
Letting $\Delta=\bigcap_{j=0}^{m-1} (\Delta_j\cap t^{-1}\Delta_j)$, we get (\ref{e-add2}) with
$C_1=\frac{\log L}{4\log(1/c)}$. 

It is tempting now to choose 
$$
\Gamma=\bigcap_{j=0}^{m-1} (\Gamma_j\cap t^{-1}\Gamma_j) \cap [0,1]
$$
with $\Gamma_j:= \frac{25^{-j}}{6}\zz+\big(0,\frac{25^{-j}\eta}{12}\big)$. Then $\Gamma-\Gamma\subset\Delta$, and since each $\Gamma_j$ has density $\eta/2$ in $\rr$ (in the obvious intuitive sense), we might expect that 
\begin{equation}\label{e-gammasize}
|\Gamma|\geq \Big(\frac{\eta}{2}\Big)^{2m}
\end{equation}
which is greater than $25^{-(1-\epsilon)m}$ as long as $\epsilon<1-\frac{\log 4-2\log\eta}{\log 25}$ (note that the last fraction is less than 1 if $\eta$ is close to 1).
In reality, this turns out to be a little bit too optimistic; however, if we instead define 
$$
\Gamma=\bigcap_{j=0}^{m-1} (\Gamma_j+\tau_j)\cap (t^{-1}\Gamma_j+\tau'_j) \cap [0,1]
$$
and average over the translations $\tau_j,\tau'_j$, we find that there is a choice of $\tau_j,\tau'_j$ such that 
(\ref{e-gammasize}) holds. Of course, the new $\Gamma$ still satisfies $\Gamma-\Gamma\subset\Delta$.



\subsection{The cyclotomic divisors of $A(x)$}
\label{sec-roots}

In order to generalize the construction in Section \ref{sec-badcyclo} to a wider class of sets, we need to study vanishing sums of roots of unity. 
Let $z_1,\dots,z_k$ be $N$-th roots of unity (not necessarily primitive). When can we have
\begin{equation}\label{lamprey}
z_1+\dots+z_k=0?
\end{equation}
This is relevant to our problem for the following reason. The construction of the set $\Gamma$ depends on the divisibility of $A,B$ by cyclotomic polynomials. Since $\Phi_s$ is irreducible, it divides $A(x)$ if and only if 
$$
A(e^{2\pi i/N})=\sum_{a\in A}e^{2\pi i a/N}=0.
$$
This is a vanishing sum of roots of unity as in (\ref{lamprey}). It is therefore in our interest to obtain effective characterizations of such sums, with $N=\text{lcm}(S_A)$ and
$$
S_A=\{s:\ \Phi_s|A,\ (s,|A|)=1\}.
$$
Fortunately, the subject has been studied quite extensively in number theory, see e.g. \cite{deB}, \cite{LL}, \cite{Mann}, \cite{PR}, \cite{Re1}, \cite{Re2}, \cite{schoen}, \cite{CJ}, \cite{CS}.

Clearly, one instance of (\ref{lamprey}) is when $k$ divides $N$ and $z_1,\dots,z_k$ are $k$-th roots of unity:
$$\sum_{j=0}^{k-1}e^{2\pi i j/k}=0.$$
Geometrically, this is represented by a regular $k$-gon on the unit circle. It is easy to construct further examples by rotating such regular polygons or adding them together. For example, adding a regular $k$-gon and a rotated regular $k'$-gon
$$
\sum_{j=0}^{k-1}e^{2\pi i j/k}+e^{2\pi i/r}\sum_{j'=0}^{k'-1}e^{2\pi i j'/k'}=0
$$
produces another vanishing sum of $N$-th roots of unity, provided that $N$ is divisible by $\text{lcm}(k,k',r)$.

A fundamental theorem of R\'edei \cite{Re1}, \cite{Re2}, de Bruijn \cite{deB} and Schoenberg \cite{schoen} asserts that in fact all vanishing sums of roots of unity can be represented in this manner, provided that we are allowed to {\em subtract} polygons as well.

\begin{theorem}\label{RdBS}
Every vanishing sum of roots of unity (\ref{lamprey}) can be represented as a linear combination of regular polygons with integer (positive or negative) coefficients.
\end{theorem} 

It is important to note that linear combinations of polygons with positive coefficients are not sufficient. For example, let
\begin{align*}
&e^{2\pi i/5}+e^{4\pi i/5}+e^{6\pi i/5}+e^{8\pi i/5}+e^{5\pi i/3}+e^{\pi i/3}\\
&=\sum_{j=0}^{4}e^{2\pi i j/5}-\sum_{j'=0}^{2}e^{2\pi i j'/3}
+(e^{2\pi i/3}+e^{5\pi i/3})+(e^{\pi i/3}+e^{4\pi i/3})
\end{align*}
This is a vanishing sum of roots of unity, represented here as (pentagon) $-$ (triangle) $+$ (2 line segments). It cannot, however, be written as a linear combination of regular polygons with positive coefficients.

What we need is an effective version of Theorem \ref{RdBS}. Essentially, we want to be able to control the number and size of the polygons used in the decomposition relative to the size of $A$. Theorem \ref{RdBS}, as it stands, does not preclude the possibility that a vanishing sum with few non-zero terms can only be represented by combining many positive and negative large polygons with massive cancellations between them. This is what we wish to avoid. 

Some quantitative results of this type are already available in the literature:

\begin{itemize}

\item (de Bruijn \cite{deB}) If $N=p^\alpha q^\beta$, with $p,q$ prime, then any vanishing sum of $N$-th roots of unity is a linear combination of $p$-gons and $q$-gons with nonnegative integer coefficients. In particular, $|A|\geq\min (p,q)$.

\item (Lam-Leung \cite{LL}) If $N=\prod p_j^{\alpha_j}$, then $|A|=\sum a_jp_j$, where $a_j$ are nonnegative integers. In particular, $|A|\geq \min \{p_j\}$.
\end{itemize}

For comparison, here is a condition from \cite{BLV} under which the construction of $\Gamma$ from Section \ref{sec-badcyclo} generalizes in a very direct manner.

\begin{proposition}\label{compatible}
Suppose that we can write $N=PQ$ with $P,Q>1$ so that: 
\begin{itemize}
\item $s$ does not divide $Q$ for any $s\in S_A$,
\item  $|A|>P$.
\end{itemize}
Suppose also that a similar statement holds for $B$. 
Then $\phi_t$ is SLV-structured for every $t\in[0,1]$.
\end{proposition}

The above suffices to prove the results in \cite{BLV} for $|A|,|B|\leq 6$.
If $|A|=2,3,4$, or $6$, we can use structural results such as Theorem \ref{RdBS} to prove that $A(x)$ cannot in fact be divisible by $\Phi_s$ with $s$ relatively prime to $|A|$. 
If on the other hand $|A|=5$, such divisors are indeed possible, as the example in Section \ref{sec-badcyclo} shows. Moreover, there may be many such divisors. We are, however, able to prove that all values of $s$ such that $\Phi_s|A$ and $(s,|A|)=1$ have the form $s=2^\alpha 3^\beta m(s)$, where $\alpha,\beta$ are the same for all $s$ and $m(s)$ is relatively prime to 6. Thus, Proposition \ref{compatible} applies with $P=2$ or 3.

Clearly, the size restrictions on $A$ and $B$ are not needed if $A$ and $B$ are given explicitly and if the assumptions of Proposition \ref{compatible} can be verified directly. 
There are, however, examples of sets for which these assumptions are not satisfied. 
To extend Theorem \ref{thm-upper} (ii) to all rational product Cantor sets (with no size restrictions on $A,B$), we would need a more general result.

In follow-up work to \cite{BLV}, Matthew Bond and I formulated the following conjecture and verified it in a number of special cases, including some where the assumptions of Proposition \ref{compatible} fail to hold. We do not know of any counterexamples.

\begin{conjecture} \label{moneygrowsontrees}
For any $A(x)$ as above, there is a $Q|N$ such that
\begin{itemize}
\item  $s$ does not divide $Q$ for any $s\in S_A$,
\item  $|A|>T$, where
$$
T=\max\{\frac{s}{(s,Q)}:\ s\in S_A\}
$$
\end{itemize}
\end{conjecture}

The conditions in Conjecture \ref{moneygrowsontrees} are sufficient to allow a construction of $\Gamma$ along lines similar to those in Section \ref{sec-badcyclo} and Proposition \ref{compatible}. Therefore, proving the conjecture would also remove the size restriction from Theorem \ref{thm-upper} (ii). Unfortunately, we do not know how to do this at this time.


\section{A Favard length estimate for random sets}
\label{PS-rewrite}

Following \cite{PS1}, we construct a ``random 4-corner set" $G_\infty=\bigcup_{n=1}^\infty$ via a randomized Cantor iteration process. 
Partition the unit square into 4 congruent squares of sidelength 1/2.  In each of these squares, choose one of the 4 dyadic subsquares of length 1/4, independently and uniformly at random. This produces the set $G_1$ consisting of 4 squares of sidelength 1/4, one in each of the 4 squares of length 1/2. We continue by induction: assume that we have already constructed the set $G_n$ consisting of $4^n$ dyadic squares of sidelength $4^{-n}$. Subdivide each of the squares of $G_n$ into 4 squares of length $4^{-n}/2$, and in each of those, choose a dyadic square of sidelength $4^{-n-1}$, uniformly and at random. All the random choices are made independently between different parent squares, and also independently of all the previous steps of the construction.

\begin{figure}[htbp]\centering\includegraphics[width=4.in]{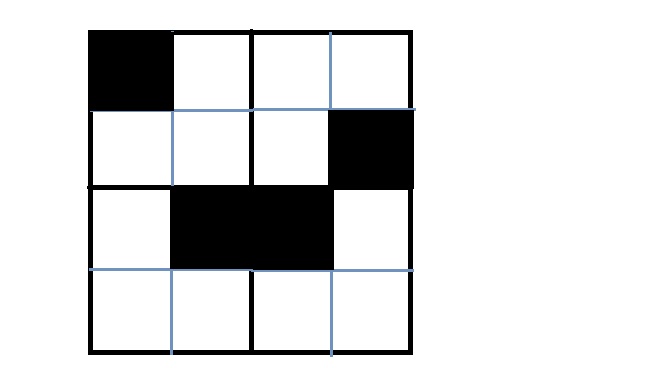}
\caption{The first iteration of a random 4-corner set. The 4 selected live squares have addresses 02, 13, 21, 33.}
\label{fig-random}
\end{figure}

\begin{theorem}\label{PS-random} \cite{PS1}
We have
\begin{equation}\label{e-ps}
\mathbb{E} ( {\rm Fav}(G_n) )\leq C n^{-1} .
\end{equation}

\end{theorem}

\begin{proof}
This is a simplified and streamlined version of the argument in \cite{PS1}.

We will use the convention that whenever a dyadic square in the plane is partitioned into 4 congruent dyadic squares, these subsquares will be labelled 0, 1, 2, 3 counterclockwise starting with the lower left corner. Thus each square $Q$ in $G_n$  has an address ${\bf g}= g_1x_1\dots g_n x_{n}$, 
where $g_i,x_i\in \{0,1,2,3\}$, defined in the obvious way. Here $g_i$ is the label of the deterministic square of sidelength $2\cdot 4^i$, and $x_i$ is the label of the random square of sidelength $4^{-i}$.

We will say that a square $Q\in G_n$ with address $\bf{g}$ is {\it essential} if each digit $g\in \{0,1,2,3\}$ appears at least $\delta n$ times in the sequence $g_1\dots g_n$, where $\delta>0$ is sufficiently small (to be fixed below) but independent of $n$. Otherwise, $Q$ 
is non-essential. 

\begin{lemma}\label{lemma-a1}
The number of non-essential squares in $G_n$ is bounded by $C4^{(1-\epsilon)n} $ for some $\epsilon=\epsilon(\delta)>0$.
\end{lemma}

\begin{proof}
It suffices to estimate the number of squares such that a fixed $g\in \{0,1,2,3\}$ appears at most $\delta n$ times in the sequence $g_i$. We bound this number by
$$
\sum_{j\leq \delta n} {n\choose j} 4^j \, 3^{n-j}
=\sum_{j\leq \delta n} {n\choose j} 4^{n}\,\left(\frac{3}{4}\right)^{n-j}
$$
Assuming that $n$ is large and $\delta < 0.1$, we can bound this by
\begin{equation}\label{a-e1}
\delta n {n\choose \lfloor \delta n \rfloor} 4^{n}\,   \left(\frac{3}{4}\right)^{n-\delta n}.
\end{equation}
By Stirling's formula,
\begin{align*}
\ln {n\choose \lfloor \delta n \rfloor }&=\ln(n!)-\ln((n-\delta n)!)-\ln((\delta n)!)
\\
&=n \ln n -(n-\delta n)\ln (n-\delta n)- \delta n \ln(\delta n)+ o(1) 
\\
&=c(\delta)n + o(1),
\end{align*}
where 
$$
c(\delta)= (1-\delta) \ln\frac{1}{1-\delta} -\delta\ln\frac{1}{\delta} \to 0
\hbox{ as } \delta\to 0.
$$
Hence for sufficiently large $n$ and small $\delta>0$,
$$
(\ref{a-e1})\leq \delta n e^{c(\delta)n+o(1)} 4^{n}\,   \left(\frac{3}{4}\right)^{n-\delta n}
\leq 4^{n}\,   \left(\frac{3}{4}\right)^{n-\delta n/2}
\leq 4^{n(1-\epsilon)}.
$$

\end{proof}

\begin{lemma}\label{lemma-a2}
Let $G_n$ be as above. 
Then for almost every slope $t$ (except
for the zero-measure set of directions given by lines that hit more than one vertex), 
whenever a line $\ell$ with slope $t$ intersects an essential square
$Q_n$, the expected number of squares of $G_n$ hit by $\ell$ is at least $\delta n/2$.
\end{lemma}

\begin{proof}
For $g\in \{0,1,2,3\}$, we will say that a square $Q_n$ of $G_n$ has type $g$ 
if $g$ appears in its address at least $\delta n$ times. It suffices to prove that
for every $t$ as in the lemma, there is a $g$ such that the conclusion of the lemma follows 
whenever $Q_n$ has type $g$. This implies the lemma, since each essential square has type $g$
for all $g\in \{0,1,2,3\}$.

Let $0<t<1$; other cases follow by symmetry. If $Q_n$ is a square of $G_n$, we will use $Q_j$ to denote the ancestor of $Q_n$ in $G_j$.

Suppose that $\ell$ intersects a type 0 square $Q_n$ with address ${\bf g}$. Then there are at least $\delta n$ scales $j$ such that $g_{j+1}=0$. For each such scale, $Q_{j}$ is a live square of $G_{j}$, and $Q_{j+1}$ lies within the lower left dyadic subsquare of $Q_{j}$ of sidelength $4^{-j}/2$.
By basic geometry, $\ell$ has to intersect at least $4^{n-j}/2$ dyadic squares of sidelength $4^{-n}$ in $Q_j\setminus Q_{j+1}$, and each of these squares has probability $4^{-(n-j)}$ of being a live square of $G_n$. Thus with probability at least 1/2, the line $\ell$ will intersect at least one live $n$-th iteration square contained
in $Q_j\setminus Q_{j+1}$. Since the sets $Q_j\setminus Q_{j+1}$ are mutually disjoint, the lemma follows.

\end{proof}

This implies (\ref{e-ps}) as follows. By Lemma \ref{lemma-a1}, the non-essential squares contribute at most
$O(4^{-\epsilon n})$ to (\ref{e-ps}). On the other hand, whenever a line in a non-exceptional direction hits an essential square, the expected number of squares it hits is at least $\delta n/2$ by Lemma \ref{lemma-a2}, so the projection of the set of essential squares in almost every direction has expected length bounded by $C/\delta n$.

\end{proof}


\bigskip
\noindent\textbf{Acknowledgement.} The author is supported in part by NSERC Discovery Grant 22R80520.


}

\bibliographystyle{amsplain}

\noindent{\sc Department of Mathematics, University of British Columbia, Vancouver,
B.C. V6T 1Z2, Canada}

\smallskip
                                                                                     
\noindent{\it  ilaba@math.ubc.ca}

\end{document}